\documentclass[11pt]{article}
\usepackage{amsfonts}
\usepackage{amssymb}
\usepackage{amsthm}
\usepackage{amsmath}
\usepackage{graphicx}
\usepackage{empheq}
\usepackage{indentfirst}
\usepackage{cite}
\usepackage{mathrsfs}
\usepackage{graphics}
\usepackage{enumerate}
\usepackage{color,xcolor}

 \newtheorem{theorem}{Theorem}[section]
 \newtheorem{corollary}{Corollary}[section]
 \newtheorem{lemma}{Lemma}[section]
 \newtheorem{proposition}{Proposition}[section]

 \newtheorem{remark}{Remark}[section]
 \numberwithin{equation}{section}

\usepackage{hyperref}					% Reference
\usepackage{cleveref}
\hypersetup{colorlinks,%				% Reference color setup	
	linkcolor=blue,%
	citecolor=blue}

\def\A{\mathcal{A}}

\allowdisplaybreaks[4]

\newcommand{\nn}{\nonumber}

\newcommand{\beq}{\begin{equation}}
\newcommand{\eeq}{\end{equation}}
 \def\non{\nonumber }
\def\bea{\begin{eqnarray}}
\def\eea{\end{eqnarray}}

\def\L{\mathcal{L}}

\begin{document}
\title{Global Existence and Uniform Boundedness in a Fully Parabolic 
Keller--Segel System with Non-monotonic Signal-dependent Motility}

\author{Yamin Xiao\thanks{The Graduate School of China Academy of Engineering 
	Physics, Beijing 100088, P.R. China, \textsl{xiaoyamin20@gscaep.ac.cn}.},
\ Jie Jiang\thanks{Innovation Academy for Precision Measurement Science 
	and Technology, CAS,
	Wuhan 430071, HuBei Province, P.R. China,
	\textsl{jiang@apm.ac.cn}.}}
\date{\today}

\maketitle

\begin{abstract} 
This paper is concerned with  global solvability of a fully parabolic system 
of Keller--Segel-type involving non-monotonic signal-dependent motility. First, 
we 
prove global 
existence of classical solutions to our problem with generic positive motility function under a
certain smallness assumption at infinity, which however permits the motility function to be  arbitrarily large  within a finite region.
Then   uniform-in-time 
boundedness of classical solutions is established whenever the motility 
function has strictly positive lower and upper bounds in any dimension $N\geq1$, 
or 
decays at a certain slow
 rate at infinity for $N\geq2$. Our results   remove the crucial
non-increasing requirement on the motility function in some recent work 
\cite{JLZ22,FJ19b,FS22} 
and hence 
allow for both 
chemo-attractive and chemo-repulsive effect, or their co-existence in applications. The 
 key 
ingredient of our
 proof lies in an important improvement  of the comparison method developed 
 in  \cite{JLZ22,FJ19b,LJ21}.

\end{abstract}	


\section{Introduction}	
In this paper, we consider global existence of classical solutions to an 
initial-Neumann boundary value problem of the following Keller--Segel system 
:
\begin{subequations}\label{cp}
\begin{align}
	&	u_t=\Delta(u\gamma(v))\,, \qquad &(x,t)\in \Omega\times (0,\infty)\,, 
	\label{cp1}\\
	&\tau v_t-\Delta v+ v=u\,, \qquad  &(x,t)\in \Omega\times 
	(0,\infty)\,,\label{cp2}\\
	& \nabla (u\gamma(v)) \cdot \mathbf{\nu} = \nabla v \cdot \mathbf{\nu} = 
	0\,, \qquad  &(x,t)\in \partial\Omega\times (0,\infty)\,, \label{cp3} \\
	& {u(x,0)  = u_0(x),\;\; v(x,0)  = v_0(x)}, \qquad 
	&x\in\Omega\,. \label{cp4}
\end{align}
\end{subequations} Here, $u$ and $v$ denote the density of cells and the signal 
concentration, respectively. $\tau>0$ is a positive constant. 
$\Omega\subset\mathbb{R}^N$ with $N\ge1$ is a  bounded  domain with smooth 
boundary. 
The above system 
features a signal-dependent motility function $\gamma(\cdot)$, which satisfies
\begin{equation}
	\gamma\in C^3((0,\infty))\,, \quad \gamma>0\, \text{ in  }\;\; 
	(0,\infty)\,. \tag{A0} \label{g1}
\end{equation}

This PDE system was originally proposed  by Keller and Segel in their seminal 
work 
\cite{KSb}, which models the oriented movement of cells due to a local sensing 
mechanism for chemotaxis. The sign of $\gamma'$ corresponds to different 
chemotaxis phenomenon: cells are attracted by signals when $\gamma'<0$, and are 
repelled when $\gamma'>0$. 

Recently, an extended three-component chemotaxis system of 
\eqref{cp1}-\eqref{cp2} has been developed in \cite{Sciencs11}, where a third 
variable $n$ denoting the nutrient level is introduced. The full system {reads as}
\begin{subequations}\label{cpn}
	\begin{align}
		& \partial_t u  = \Delta (u \gamma(v)) + \theta u f(n)\,, \qquad  
		&(x,t)\in \Omega\times (0,\infty)\,, \label{cpn1} \\
		&	\tau \partial_t v=\Delta v -  v  +  u\,, \qquad  &(x,t)\in 
		\Omega\times (0,\infty)\,, \label{cpn2} \\
		& \partial_t n= \Delta n - \theta u f(n) \,, \qquad  &(x,t)\in 
		\Omega\times (0,\infty)\,. \label{cpn3} 
	\end{align}
\end{subequations}
In their work, $\gamma$ is assumed to be a positive non-increasing function, which stands for a repressive effect of the signal on cellular motility. By numerical and experimental analyses, it is  shown that the system can foster spatially periodic patterns in a growing bacteria population merely under this motility control. We note that setting $\theta=0$ in \eqref{cpn} cancels the coupling and hence we recover system \eqref{cp1}-\eqref{cp2} for $(u,v)$.

The mathematical analysis for the above Keller--Segel systems with 
signal-dependent motility has attracted a lot research interest in recent
years. We refer the readers to \cite{JLZ22} for a detailed review on this 
topic where the non-increasing monotonicity of $\gamma$  is mostly assumed. {For general $\gamma$ satisfying \eqref{g1}, $\gamma'\leq0$ and $\lim\limits_{s\rightarrow\infty}\gamma(s)<1/\tau$, global existence of unique classical solution  to problem \eqref{cp}  was established in \cite{FJ19b} when $N=2$, and in \cite{FS22} when $N\geq3$.}

 In 
contrast, existence theories without the monotonicity requirement are largely 
missing in the literature. On the one hand, if $\gamma$ has strictly positive bounds from 
above and 
  below, and moreover $|\gamma'|$ is bounded, Tao and Winkler \cite{TaoWin17} 
  prove that there exists a unique globally bounded classical solution for 
  $N=2$, and global weak solution is constructed for $N\geq3$ provided that $\Omega$ is convex. If the motility 
  is 
  permitted to approach zero at infinity, global existence of (very) weak 
  solutions is established if $\gamma$ either  satisfies  
  $\sup\limits_{s>0}\gamma(s)<1/\tau$ in \cite{LJ21}, or exhibits an algebraic 
  decay at infinity\cite{DLTW22}.
  
  On the other hand, it is worth noting that for the 
  parabolic-elliptic 
  version of \eqref{cp} with $\tau=0$, existence of classical solution is 
  verified  simply under the assumption \eqref{g1} and that 
  $\sup\limits_{s\in[a,\infty)}\gamma(s)<\infty$ for all $a>0$ (or equivalently, $\limsup\limits_{s\rightarrow\infty}\gamma(s)<\infty$), without the need of  non-increasing 
property. However, a corresponding theory for the fully parabolic case is still open. Thus, there arises a natural question that whether the monotonicity requirement is removable or replaceable to ensure the 
  existence of global classical solutions for the fully parabolic system.

  The first task of this work is to give an affirmative answer to the above question. More precisely, we establish the following existence result.
%%%%%%%%%%%%%%%%%%%%%%%%%%%%%%%%%%%%
%%%%%%%%%%%%%%%%%%%%%%%%%%%%%%%%%%%%%%%%%
\begin{theorem}\label{TH0} Let $N\ge 1$. Suppose that $\gamma$ satisfies assumptions~\eqref{g1}, and that the initial condition $(u_0, v_0)$ satisfies 
\begin{equation}\label{ini}
(u_0, v_0)\in C^0(\bar\Omega)\times W^{1,\infty}(\Omega),\quad u_0\geq0,\;v_0>0 
\quad  \mbox{in } \bar\Omega, \quad u_0\not\equiv0.\end{equation}
Moreover, we assume 
\begin{equation}\label{gsup}\tag{A1}
	\gamma_\infty:=\limsup\limits_{s\rightarrow\infty} \gamma(s)<1/\tau.
\end{equation}
Then problem~\eqref{cp} has a unique global non-negative classical solution 
$(u,v)\in C( \bar{\Omega}\times [0,\infty);\mathbb{R}^2) \cap C^{2,1}( 
\bar{\Omega}\times (0,\infty);\mathbb{R}^2)$.
\end{theorem}
%%%%%%%%%%%%%%%%%%%%%%%%%%%%%%%%%%%%%%%%%%%%
\begin{remark}
Note  the above result now permits sign-changing of $\gamma'$ and thus applies 
for both chemo-attraction and chemo-repulsion.

{ In \cite{LJ21}, global existence of weak solutions was proved where sign-changing of $\gamma'$ is also permitted. However, a crucial condition $\sup\limits_{s\in(0,\infty)}\gamma(s)<1/\tau$ is assumed there which requires uniform smallness of $\gamma$ on $(0,\infty)$. In contrast, a much weakened
	assumption \eqref{gsup} in the current work  is proposed where smallness is merely presumed near infinity and hence $\gamma$ can be 
	arbitrarily large within any finite interval.}
\end{remark}

\medskip

Our second task is to study the dynamic of global classical solutions. More precisely, we are interested in the boundedness of global solutions to problem \eqref{cp}. If $\gamma$ has strictly positive lower and 
upper  
bounds, i.e., 
\begin{equation}\label{gb}\tag{A2}
\text{there are generic constants} \;\gamma_*,\gamma^*>0\;\text{such 
that}\;\gamma_*\leq\gamma(\cdot)\leq\gamma^*\;\text{in}\;(0,\infty),
\end{equation} there are limited evidences indicating that 
the system  fosters bounded solutions. In 
this direction, the first result is the one we mentioned before by Tao and Winkler in \cite{TaoWin17}, proving that when 
$N=2$, problem \eqref{cp} permits a unique globally bounded classical solution on a convex domain
 provided that $\gamma$ satisfies \eqref{gb} and $|\gamma'|$ is also bounded. The other result in higher dimension $N\geq 3$ is provided in \cite[Remark 
1.3]{JL21} for the parabolic-elliptic  simplification of system \eqref{cp} 
($\tau=0$), establishing the uniform boundedness when $\gamma$ is 
non-increasing and has  strictly positive lower and upper bounds. 
%
%No further boundedness results are 
%available in the case $\tau>0$ and $N\geq3$ when $\gamma$ simply satisfies 
%\eqref{g1} and \eqref{gb}.
%%%%%%%%%%%%%%%%%%%%%%%%%%%%%%%%%%%%%%%%%%%%%

The second main result of this paper considers  the case when $\gamma$ simply 
satisfies \eqref{g1} and \eqref{gb}. We establish global existence of uniformly 
bounded classical solutions in any dimension. {Here and after, we use the short notation $\|\cdot\|_p$ for the norm $\|\cdot\|_{L^p(\Omega)}$ with $p\in[1,\infty]$.}
\begin{theorem}
	\label{TH1}
Let $N\ge 1$. Suppose that $\gamma$ satisfies assumptions~\eqref{g1} and 
\eqref{gb}, and that the initial condition $(u_0, v_0)$ satisfies \eqref{ini}. 
Then problem~\eqref{cp} has a unique global non-negative classical solution 
that is uniformly-in-time bounded, i.e.,
\begin{equation}
	\sup_{t\ge 0}\left\{ \|u(\cdot,t)\|_\infty + \|v(\cdot,t)\|_{W^{1,\infty}} \right\} < 
	\infty\,. \label{ubd}
\end{equation}
\end{theorem}
\begin{remark}
	 Note that we do not need assumption \eqref{gsup} here. Our result improves  the 2D boundedness result in \cite[Theorem 1.1]{TaoWin17} by establishing boundedness in any dimension and removing the  boundedness assumption on $|\gamma'|$, as well as the convexity of domain as requested there.
\end{remark}
\begin{remark}
	In view of the time-independent positive lower bound for $v$ that are also independent of the choice of $\gamma$, see  \eqref{e00} below, assumption \eqref{gb} can be slightly weakened as 
	\begin{equation*}
	\text{there are generic constants} \;\gamma_*,\gamma^*>0\;\text{such 
		that}\;\gamma_*\leq\gamma(\cdot)\leq\gamma^*\;\text{in}\;[v_*,\infty).	
	\end{equation*}
\end{remark}
\medskip

Boundedness results are also available in the literature if $\gamma$ is non-increasing and 
tending 
to  zero at infinity, i.e.,
\begin{equation}\label{gdecay}\tag{A3}
	\gamma'\leq0,\qquad\text{and}\;\lim\limits_{s\rightarrow \infty}\gamma(s)=0.
\end{equation}
Thus the motility function has no strictly positive lower bound and the above asymptotic vanishing property
leads to a  possible degenerate problem. Under the circumstances, dynamic 
of global solutions is shown to be closely related to the decay rate of $\gamma$ at 
infinity. 
Exponential decay rate is verified to be critical  if $N=2$  in a sense that 
the solution can become unbounded at time infinity when 
$\gamma(s)=e^{-s}$ with some large mass initial data \cite{FJ19a,FJ19b,JW20}, 
whereas it 
is always uniformly bounded when 
$\gamma$ decays slower than a negative exponential function, i.e., $\gamma$ 
satisfies
\begin{equation}\label{growth2da}
	\liminf\limits_{s\rightarrow\infty}e^{\chi s}\gamma(s)>0
\end{equation} for all $\chi>0$ \cite{FJ20a, 
JLZ22}, e.g., $\gamma(s)=s^{-k_1}\log^{-k_2}(1+s)$ with any $k_1,k_2>0$, or $\gamma(s)=e^{-s^{\alpha}}$ with any $\alpha\in(0,1)$. If $N\geq3$, the global classical solution is proved to be uniformly-in-time bounded in \cite{JL21,JLZ22, FS22b},  provided that $\gamma$ exhibits an algebraic decay at infinity. More precisely, the following algebraic growth assumption is requested on $1/\gamma$:
\begin{equation}\label{gamma2}
\text{there are $k\geq l\geq0$ such that}\;\;\liminf\limits_{s\rightarrow\infty}s^{k}\gamma(s)>0\;\;\text{and}\;\;
\limsup\limits_{s\rightarrow\infty}s^{l}\gamma(s)<\infty. \tag{A4}
\end{equation}A specific example satisfying \eqref{gamma2} is $\gamma(s)=s^{-k}$. For the parabolic-elliptic version of system \eqref{cp} ($\tau=0$), it is shown in \cite{JL21} that the global solution is bounded if we require further that 
\begin{equation}\label{gamm3}
	k < \frac{N}{N-2} \;\;\text{and }\;\; k-l < \frac{2}{N-2}\,.
\end{equation} See also \cite{Anh19,Wang20} for the specific case $\gamma(s)=s^{-k}$ with $k<\frac{2}{N-2}$. For the fully parabolic case $\tau>0$, however, an additional technical assumption below is also needed in \cite{JLZ22}:
\begin{equation}\label{A5}
	\begin{split}
		&\text{there is}\;  b_0\in (0,1] \;\text{such that, for any $s\ge s_0>0$,} \\
		& \hspace{2cm} s\gamma(s)+(b_0-1) \int_1^s\gamma(\eta)\mathrm{d}\eta\leq K_0(s_0) \,, \\
		&\text{where $K_0(s_0)>0$ depends only on $\gamma$, $b_0$, and $s_0$}\,.
	\end{split}\tag{A5}
\end{equation}

One another aim of the present paper is to revisit the boundedness result in \cite{JLZ22,FJ20a} 
for $N\geq2$. { Thanks to an improvement of the comparison argument, 
we are now 
able to remove the non-increasing 
monotonicity and  the above technical assumption \eqref{A5} on $\gamma$ under the same decay rate assumption as in  \cite{JLZ22,FJ20a}.}

Now we rephrase the boundedness result for problem 
\eqref{cp} with possible decaying motility as follows.
\begin{theorem}\label{TH2a}
	Let $N=2$. Suppose that $\gamma$ satisfies  assumptions~\eqref{g1} and \eqref{gsup}. Moreover, there is $\chi>0$ such that \eqref{growth2da} is satisfied. If the initial condition $(u_0,v_0)$ satisfies \eqref{ini} with 
{	\begin{equation}
		\|u_0\|_1<\frac{4\pi(1-\tau\gamma_\infty)}{\chi},
	\end{equation}
}then the global classical solution $(u,v)$ to \eqref{cp} is uniformly-in-time bounded.

In particular, if \eqref{growth2da} is satisfied for all $\chi>0$, then the global classical solution to \eqref{cp} is uniformly-in-time bounded for any given initial data satisfying \eqref{ini}. 
\end{theorem}
{
	\begin{remark}
	If $\lim\limits_{s\rightarrow\infty}\gamma(s)=0$, then $\gamma_\infty=0$ and  our result ensures boundedness of classical solutions if $\|u_0\|_1<4\pi/\chi$, or \eqref{growth2da} is satisfied for all $\chi>0$, which improves previous results in \cite{FJ20a,JLZ22} by removing the monotonicity assumption. In this case, the decay rate  \eqref{growth2da} is optimal since infinite-time blowup occurs if $\gamma(s)=e^{-\chi s}$ with some initial data of total mass exceeding $4\pi/\chi$, see \cite{FJ19a,FJ19b,JW20}.
\end{remark}
}
\begin{theorem}\label{TH2}
	 Let $N\ge 3$. Suppose that $\gamma$ satisfies assumptions~\eqref{g1} and  \eqref{gsup}. Moreover, assumption \eqref{gamma2} is fulfilled with some $k\geq l\geq0$ satisfying \eqref{gamm3}.

Then, for any initial condition $(u_0,v_0)$ satisfying  \eqref{ini}, the  global 
classical solution $(u,v)$  to \eqref{cp} is uniformly-in-time bounded in the 
sense that it satisfies \eqref{ubd}.
\end{theorem}
\begin{remark}
In Theorem \ref{TH2}, assumption \eqref{gsup} 	is only explicitly needed in the case $l=0$ and $0<k<\frac{2}{N-2}$; see the proof of Lemma \ref{lemineq}.
\end{remark}
{ 
\begin{remark}If  $\gamma(s)=s^{-k}$, we note that \eqref{cp1}-\eqref{cp2} can be written as 
	\begin{equation}\label{log}
		\begin{cases}
				u_t=\nabla\cdot\left(\gamma(v)\nabla u+u\gamma'(v)\nabla v\right)=\nabla\cdot \bigg(\gamma(v)(\nabla u-ku\nabla \log v)\bigg)\\
				\tau v_t-\Delta v+v=u.
		\end{cases}
	\end{equation}
System \eqref{log} resembles  the logarithmic Keller--Segel system and they share the same steady states. For the latter system, it was conjectured in \cite{FS18} that  $k_c=\frac{N}{N-2}$ is the threshold number that  separates boundedness and
blowup of classical solutions. 

According to Theorem \ref{TH2}, we achieve boundedness of global classical solutions to \eqref{log} in the sub-critical case $k<k_c$, whereas for the logarithmic Keller--Segel system, boundedness results in the sub-critical region are still far from satisfaction.
\end{remark}   }

\medskip
The main idea of our proof relies on  a significant improvement of the comparison argument 
of a systematic approach developed recently in 
\cite{FJ19a,FJ19b,LJ21,JL21,JLZ22,FS22b,FS22} for  systems \eqref{cp} and 
\eqref{cpn}. 
Since the approach is quite different from the conventional energy method {(e.g., \cite{TaoWin17,JW20,BLT20,Anh19})}, let 
us outline its procedure here. The key ingredient lies in an 
introduction 
of an auxiliary function $w(x,t)$, which satisfies an elliptic problem:
\begin{equation*}
	\begin{cases}
		w-\Delta w=u \,, \qquad  &(x,t)\in \Omega\times (0,\infty)\,, \\
		\nabla w\cdot\nu=0\,, \qquad &(x,t)\in \partial\Omega\times (0,\infty)\,.
	\end{cases}
\end{equation*}
 By taking an inverse operator $\A=(I-\Delta)^{-1}$ on both sides of  \eqref{cp1} with $\Delta$ being the Laplacian operator with homogeneous Neumann boundary condition, one obtains the following key identity:
\begin{equation}\label{keyid1}
	w_t+u\gamma(v)=\A^{-1}[u\gamma(v)].
\end{equation}
This identity is uncovered in \cite{FJ19b} and has since then been used efficiently to investigate the
global existence and boundedness of classical solutions to \eqref{cp} and its variants \eqref{cpn}, see \cite{FJ19a,FJ19b,FJ20a,JL21,JLZ22,LJ21,DLTW22,LyuWang1,LyuWang2}.

Substituting $u=w-\Delta w$ into the key identity, we transform the problem to a 
thorough study of the 
following quasilinear parabolic equation involving a non-local term:
\begin{equation}\label{keyid2}
	w_t-\gamma(v)\Delta w+\gamma(v)w=\A^{-1}[u\gamma(v)].
\end{equation}
{The proof for global existence and boundedness mainly consists of four steps as we sketch below.
\begin{itemize}
\item  Step I: estimates of $w$ and $v$ in $L^\infty(0,T;L^\infty(\Omega))$ for arbitrary $T>0$, with a possible dependence on $T$.  For this purpose, a three-step comparison argument is developed. 
\begin{enumerate}[(a)]
	\item Firstly, the above key identity \eqref{keyid1} allows us to derive a 
	(time-dependent) upper bound of $w$ via comparison principle of elliptic 
	equations; see Lemma \ref{keylem} below;
	\item  Secondly, an application of comparison principle 
	for heat equations yields to an upper control of $v$ by $w$ and thus gives 
	a 	(time-dependent) upper bound of $v$;
	\item Thirdly, we use comparison principle of 
	elliptic as well as heat equations again to prove a reverse control. Hence 
	we 
	establish a two-sided control of $v$ by $w$.
\end{enumerate} 
\item Step II: time-independent upper bounds for $w$ and $v$. 
\begin{enumerate}[(a)]
	\item If $N\leq2$, the uniform-in-time upper bounds are derived via a simple way relying on a crucial time-independent duality estimate and Sobolev embeddings.
	\item  When $N\geq3$, the proof is more tricky. Thanks to the two-sided control and the monotonicity of $\gamma$, 
	we can replace 
	$\gamma(v)$ by $\gamma(Cw)$ in \eqref{keyid2} with some generic positive constant 
	$C$. Then a delicate Moser-Alikakos 
	iteration is applied to derive the uniform-in-time upper bounds for $w$ and $v$. 
\end{enumerate}

\item Step III:  H\"older estimate for $v$ in $C^\alpha(\bar{\Omega}\times[0,T])$ with arbitrary $T>0$, with a possible dependence upon $T$ of both the estimate and the exponent $\alpha\in(0,1)$. We establish a local energy estimate which implies that $w$  is bounded in $C^\alpha(\bar{\Omega}\times[0,T])$ by the classical result in \cite{LSU}. The H\"older estimate for $v$ then follows by standard Schauder's theory.

\item Step IV:  estimates of $(u,v)$ in $L^\infty(0,T; L^\infty(\Omega)\times W^{1,\infty}(\Omega))$, still possibly depending on $T$. 
The above H\"older continuity enables us to  regard $-\gamma(v(t))\Delta$ as a generator of an evolution operator. Thus we can apply the abstract semigroup theory in \cite{Amann} to deduce the high-order estimates for $w$. Finally, we use energy method to get the  
$W^{1,\infty}$-boundedness 
of $v$, and $L^\infty$-boundedness of $u$, respectively.
\end{itemize}
}
{ For global existence of classical solutions, the proof goes along Step I-(a, b), Step III and Step IV. For uniform-in-time boundedness, the proof consists of  Step I-(b), Step II-(a), Step III and Step IV when $N\leq2$, and it follows Step I-(b, c), Step II-(b), Step III and Step IV when $N\geq3$. 
}

{
	We point out that  monotonicity of $\gamma$ is not necessary in 
	Step III and Step IV once we obtain the (time-dependent or time-independent) upper bounds of $w$ 
	and $v$. However, the non-increasing property plays a crucial role in comparison argument part Step 
	I-(b,c), which restrict the previous studies \cite{FJ19a,FJ19b,JL21,JLZ22,FS22,FS22b} mainly focusing on the chemo-attraction case.
%	the reverse control obtained in Step I-(c) is only needed 
%when we study uniform boundedness in the case $\tau>0$ and 
%$N\geq3$. Indeed, if $\tau=0$, $w$ is identical to $v$ by definition. While for $N\leq2$, uniform boundedness is achieved by a simple way along Step II-(a) without the need of the reverse control in Step I-(c). 
}

{We stress again that the main contribution of the present work is to completely remove the monotonicity assumption on the motility function, which achieves a crucial technical improvement of the above mentioned method and an important progress in the existence/boundedness theory for the Keller--Segel system \eqref{cp}. Moreover, it has significance in physical applications, since now $\gamma'$ is permitted to change its sign, both chemo-attraction and chemo-repulsion are included.
	
		As we analyzed above, the cornerstone of our proofs is  the derivation of $L^\infty$-bounds for $w$ and $v$  without monotonicity, which relies on an essential
improvement and a careful revisit of 
the above mentioned steps.} Firstly, we develop a variant version of the 
second comparison argument in Step I-(b) which completely  removes the non-increasing assumption; see Lemma \ref{comparisonA} below. Secondly,  under the assumption of strictly lower and 
upper boundedness of $\gamma$, we can directly apply an iteration argument on 
\eqref{keyid1} together with comparison techniques to derive the uniform boundedness of $w$ without the 
monotonicity 
assumption on $\gamma$, which allows us to skip Step II. Lastly, we observe that 
under the assumptions \eqref{gsup} and \eqref{gamma2} with  some $k\geq l\geq 0$,
monotonicity in Step I-(c) and Step II can also be dropped as well.

The paper is organized as follows. In Section 2, we recall some preliminary 
results. In Section 3, we provide the key variant proof for the upper control of $v$ 
by $w$, which removes the monotonic non-increasing property
 and enables us to prove Theorem \ref{TH0}. In Section 4, we use an iteration together 
with a comparison argument to derive the uniform upper bound for $w$ directly and then 
prove Theorem \ref{TH1}. In the last section, we discuss the boundedness when $\gamma$ satisfies some decay condition at infinity.

%%%%%%%%%%%%%%%%%%%%%%%%%%%%%%%%
%%%%%%%%%%%%%%%%%%%%%%%%%%%%%%%%
\section{Preliminaries}
In this section, we recall some basic results for problem \eqref{cp}. First, we 
have the  local well-posedness. We refer the reader to \cite[Proposition 
2.1]{JLZ22} for a detailed proof.
 \begin{theorem}\label{local}
	 Suppose that $\gamma$ satisfies \eqref{g1} and $(u_0,v_0)$ satisfies 
	 \eqref{ini}. Then there exists $T_{\mathrm{max}} \in (0, \infty]$ such 
	 that 
	 problem~\eqref{cp} has a unique non-negative classical solution $(u,v)\in 
	 C(\bar{\Omega}\times [0,T_{\mathrm{max}});\mathbb{R}^2)\cap 
	 C^{2,1}(\bar{\Omega}\times (0,T_{\mathrm{max}});\mathbb{R}^2)$. The 
	 solution $(u,v)$ satisfies the mass conservation
	\begin{equation}
	\int_\Omega u(x,t)\ \ d x= \int_\Omega u_0(x)\ d x
	\quad \text{for\ all}\ t \in (0,T_{\mathrm{max}})\,.\label{e0}
	\end{equation}
Moreover, there is $v_*>0$ depending only on $\Omega$, $v_0$, and $\|u_0\|_1$ such that
\begin{equation}
	v(x,t) \ge v_*\,, \qquad (x,t)\in	\bar{\Omega}\times  
	[0,T_{\mathrm{max}})\,. \label{e00}
\end{equation}
Finally, if $T_{\mathrm{max}}<\infty$, then
	\begin{equation*}
	\limsup\limits_{t\nearrow T_{\mathrm{max}}}\|u(\cdot,t)\|_{\infty}=\infty.
	\end{equation*}
\end{theorem}

The following result implies that under the assumptions \eqref{g1} and 
\eqref{gamma2}, $\gamma$ is controlled from the above and below by two 
algebraically decay functions. The proof is trivial and we omit the details here.
\begin{lemma}\label{gamma2b}
	Suppose that $\gamma$ satisfies \eqref{g1} and \eqref{gamma2}. Then for any $a>0$, there are constants $C_1,C_2>0$ depending on $a$ and $\gamma$ such that 
	\begin{equation}\label{gamma2c}
		C_1s^{-k}\leq \gamma(s)\leq C_2s^{-l},\qquad s\in[a,\infty).
	\end{equation}
\end{lemma}
%\begin{proof}
%By the first condition of \eqref{gamma2}, we infer there is $s_0>a$ and $C_1'>0$ such that for $s\geq s_0$, there holds
%\begin{equation*}
%	s^{-k}/\gamma(s)\leq C_1'.
%	\end{equation*}
%In view of the assumption \eqref{g1}, we can find $C_1''>0$ such that for 
%$s\in[a,s_0]$, $1/\gamma(s)\leq C_1''$ and hence
%\begin{equation*}
%	s^{-k}/\gamma(s)\leq C_1''a^{-k}\qquad\text{for}\;s\in[a,s_0].
%\end{equation*}
%Then it follows that for all $s\in[a,\infty)$,
%\begin{equation}
%	s^{-k}/\gamma(s)\leq C_1'+ C_1''a^{-k}:=1/C_1.
%\end{equation}This proves the left-hand side of \eqref{gamma2c}.
%
%In the same manner, by the second condition of \eqref{gamma2}, there is $s_1>a$ and $C_2'>0$ such that for all $s\geq s_1$,
%\begin{equation*}
%	s^{l}\gamma(s)\leq C_2'
%\end{equation*}and for $s\in[a,s_1]$, we can find $C_2''>0$ by \eqref{g1} such that
%\begin{equation*}
%	s^{l}\gamma(s)\leq C_2''s_1^{l}.
%\end{equation*}
%This concludes the proof by letting $C_2=C_2'+C_2''s_1^{l}$.
%\end{proof}

Next, we define the operator $\A$ on $L^2(\Omega)$ as
\begin{equation*}
	\mathrm{dom}(\A)\triangleq \{z\in H^2(\Omega):\qquad \nabla z\cdot \nu=0\;\;\text{on}\;\partial\Omega\},\qquad \A z\triangleq z-\Delta z,\quad z\in \mathrm{dom}(\A).
\end{equation*}
We recall that $\A$ generates an analytic semigroup on $L^p(\Omega)$ and is invertible on $L^p(\Omega)$ for all $p\in(1,\infty)$, {see e.g., \cite{Amann}}. Then we introduce $w(\cdot,t)\triangleq\A^{-1}[u(\cdot,t)]\geq0$, for $t\in[0,T_{\mathrm{max}})$, the non-negativity of $w$ being a consequence of that of $u$ and the comparison principle. Due to the time continuity of $u$, 
\begin{equation}
	w_0(x)\triangleq w(x,0)= \A^{-1}[u_0(x)]
\end{equation} and it follows the regularity assumption \eqref{ini} that $w_0\in W^{2,p}(\Omega)$ with any $1\leq p<\infty$. 

%%%%%%%%%%%%%%%%%%%%%%%%%%%%%%%%%

%%%%%%%%%%%%%%%%
We also need the following result given in \cite[Lemma~3.3]{Wang20}, which is similar to the celebrated Brezis-Merle inequality \cite[Theorem~1]{BrMe1991}, see  \cite[Lemma~A.3]{TaoWin17} for related results. 

%%%%%%%%%%%%%%%%
\begin{lemma}\label{lm2e}
	Assume that $N=2$. For any  $f\in L^1(\Omega)$ such that
	\begin{equation*}
		\|f\|_1= \Lambda>0
	\end{equation*}
	and $0<R<\frac{4\pi}{\Lambda}$, there is $C(\Lambda,R)>0$ depending on $\Omega$, $\Lambda$, and $R$ such that the solution $z$ to 
	\begin{equation}\label{helm}
		\begin{cases}
			-\Delta z+ z=f,\qquad x\in\Omega\,,\\
			\nabla z\cdot \nu=0\,,\qquad x\in\partial\Omega\,,
		\end{cases}
	\end{equation}
	 satisfies
	\begin{equation*}
		\int_\Omega e^{Rz} d x\leq C(\Lambda,R)\,.
	\end{equation*}
\end{lemma}
%%%%%%%%%%%%%%%%%%%%%%%%%%%%%%%%%

Last we recall the following key lemma. See \cite[Lemma 5]{FJ19b} for a detailed proof.
\begin{lemma}\label{keylem}Assume that $\gamma$ satisfies  \eqref{g1}. Then the 
following key identity holds
\begin{equation}\label{keyid}
	\partial_t w+u\gamma(v)=\A^{-1}[u\gamma(v)]
\end{equation}	in $\Omega\times (0,T_{\mathrm{max}})$.
Moreover, if there is a constant $\gamma^*>0$ such that $\gamma(v(x,t))\leq \gamma^*$ for $(x,t)\in\Omega\times[0,T_{\mathrm{max}})$, then there holds
\begin{equation}\label{boundw}
	w_*\leq w(x,t)\leq w_0(x)e^{\gamma^*t},\qquad\text{for 
	}\;\;(x,t)\in\Omega\times[0,T_{\mathrm{max}})
\end{equation}with $w_*$ being a positive constant depending only on $\Omega$ and $\|u_0\|_1$.
	
\end{lemma}

\section{Global existence with non-monotonic motility}
In this section, we aim to establish global existence of classical solutions to 
problem \eqref{cp} under the assumptions of Theorem \ref{TH0}. To begin with, 
we provide  an upper control of $v$ by $w$ under the assumption \eqref{gsup}, 
which enables us to consider problem \eqref{cp} without monotonicity assumption 
on $\gamma$ in this work.

First, we notice that under the assumptions \eqref{g1} and \eqref{gsup}, there 
is a positive 
constant $\gamma^*>0$ such that $\gamma(s)\leq \gamma^*$ for 
$s\in[v_*,\infty)$ with $v_*$ being the constant given in \eqref{e00}. Then it 
results from Lemma \ref{keylem} that
\begin{equation}\label{wb0}
	w_*\leq w(x,t)\leq w_0(x)e^{\gamma^* t}\qquad\text{for 
	}\;\;(x,t)\in \Omega\times [0,T_{\mathrm{max}}).
\end{equation}
\begin{lemma}\label{comparisonA}
	Under the assumption of Theorem \ref{TH0}, for any $\varepsilon>0$,
	there is a  constant $C_\infty>0$ depending only on $\Omega$, $\gamma$, $\tau,\varepsilon$ 
	and 
	the initial data such that for all 
	$(x,t)\in\bar{\Omega}\times [0,T_{\mathrm{max}})$,
{	\begin{equation}\label{eupctrol}
		v(x,t)\leq \left(\frac{1}{1-\tau\gamma_\infty}+\varepsilon\right)(w+C_\infty).
	\end{equation}
}
\end{lemma}
\begin{proof} Recall that
	$\limsup\limits_{s\rightarrow\infty}\gamma(s)=\gamma_\infty<1/\tau$ and introduce 
	\begin{equation*}
		\varepsilon_0=(1-\tau\gamma_\infty)\left(1-\frac{1}{1+\varepsilon(1-\tau\gamma_\infty)}\right)\in(0,1-\tau\gamma_\infty).
	\end{equation*}	
Since $w,v$ are both non-negative, 
	we may pick $C_\infty>0$ sufficiently large such that 
	\begin{equation}\label{gb2}
		\gamma(\sigma v+w+C_\infty)\leq\gamma_\infty+\varepsilon_0/\tau,\qquad 
		\forall \;\sigma>0
	\end{equation}as well as 
	\begin{equation}\label{ini2}
		v_0(x)\leq w_0(x)+C_\infty.
	\end{equation}
	
Denote $\L[z]=\tau z_t-\Delta z+z=\tau z_t+\A z$	and observe that 
$\A^{-1}[u\gamma(v)]\geq0$ by 
comparison principle. We infer from \eqref{cp2}, the definition of $w$, the key 
identity 
\eqref{keyid} and \eqref{gb2} that
	\begin{align}
		\mathcal{L}[v]=u	=&w-\Delta w=\L[w]-\tau w_t\non\\
		=&\mathcal{L}[w]+\tau u\gamma(v)-\tau 
		\mathcal{A}^{-1}[u\gamma(v)]\nn\\
		\leq &\mathcal{L}[w]+\tau u\bigg(\gamma(v)-\gamma(\sigma 
		v+w+C_\infty)\bigg)+\tau u\gamma(\sigma v+w+C_\infty)\nn\\
		\leq& \mathcal{L}[w]+\tau u\bigg(\gamma(v)-\gamma(\sigma 
		v+w+C_\infty)\bigg)+(\tau\gamma_\infty+\varepsilon_0)u,
	\end{align}
	which entails that
	\begin{align}\label{com}
		&(1-\tau\gamma_\infty-\varepsilon_0)u\nonumber\\
		&\leq \mathcal{L}[w]+\tau 
		u\bigg(\gamma(v)-\gamma(\sigma v+w+C_\infty)\bigg)\nn\\
		&= \mathcal{L}[w]+\bigg(\tau 
		u\int_0^1\gamma'(\theta v+(1-\theta)(\sigma 
		v+w+C_\infty))d\theta\bigg)\big((1-\sigma)v-w-C_\infty\big)\nn\\
		&\leq  \mathcal{L}[w+C_\infty]+\bigg(\tau 
		u\int_0^1\gamma'(\theta v+(1-\theta)(\sigma 
		v+w+C_\infty))d\theta\bigg)\big((1-\sigma)v-w-C_\infty\big).
	\end{align}
	Now, we let $1-\sigma=1-\tau\gamma_\infty-\varepsilon_0\in(0,1)$ and notice 
	that 
	$(1-\tau\gamma_\infty-\varepsilon_0)u=\mathcal{L}[(1-\tau\gamma_\infty-\varepsilon_0)v]=\mathcal{L}[(1-\sigma)
	v]$. It follows from \eqref{com} that
	\begin{align}
		\mathcal{L}[(1-\sigma)v-w-C_\infty]\leq\bigg(\tau 
		u\int_0^1\gamma'(\theta v+(1-\theta)(\sigma 
		v+w+C_\infty))d\theta\bigg)\big((1-\sigma)v-w-C_\infty\big).	
		\end{align}
{
	Denote
\begin{equation*}
	g(x,t):=\tau u\int_0^1\gamma'(\theta v+(1-\theta)(\sigma v+w+C_\infty))d\theta
\end{equation*} which belongs to $C(\bar{\Omega}\times [0,T_{\mathrm{max}}))$ according  to Theorem \ref{local}	and let $h(x,t):=(1-\sigma)v-w-C_\infty$. The preceding inequality now reads as
\begin{equation}
	\tau\partial_t h-\Delta h+(1-g)h\leq0.
\end{equation}
Moreover, we note that $\nabla h\cdot\nu=0$ on $\partial\Omega$, and $h(x,0)=(1-\sigma)v_0-w_0-C_\infty\leq v_0-w_0-C_\infty\leq 0$ due to \eqref{ini2}.	For any fixed $T\in(0,T_{\mathrm{max}})$, since $1-g\in C_b(\bar{\Omega}\times[0,T])$, we may apply comparison principles of linear parabolic equations (see, e.g., \cite[Theorem 2.2, Chapter 1]{LSU}) to deduce that 
	that  $h(x,t)\leq0$ for all $(x,t)\in\bar{\Omega}\times[0,T]$. As a result, we get
	\begin{equation*}
		(1-\sigma)v(x,t)\leq w(x,t)+C_\infty,\qquad\text{for 
		}\;\;(x,t)\in\bar{\Omega}\times[0,T_{\mathrm{max}}).
	\end{equation*} This completes the proof by substituting the definition of $\varepsilon_0$.
}
\end{proof}

As a result, we obtain an upper 
control of $v$ by $w$ as well as a time-dependent upper bound of $v$.
\begin{corollary}\label{cor1}
	Under the assumption of Theorem \ref{TH0}, there is a  constant  $B>0$ 
	depending only on $\Omega$, $\gamma$, $\tau$ 
	and 
	the initial data such that
\begin{equation}\label{upcontrol}
	v(x,t)\leq Bw(x,t),\qquad\text{for 
		}\;\;(x,t)\in\Omega\times[0,T_{\mathrm{max}})
\end{equation}	
	and \begin{equation}\label{upbv}
		v(x,t)\leq \frac{2(w_0e^{\gamma^* 
		t}+C_\infty)}{1-\tau\gamma_\infty},\qquad\text{for 
		}\;\;(x,t)\in\Omega\times[0,T_{\mathrm{max}}).
	\end{equation}
\end{corollary}
{
\begin{proof}
	Fixing any $\varepsilon>0$ in Lemma \ref{comparisonA}, then there is a positive time-independent constant $C>0$ such that
	\begin{equation*}
		v(x,t)\leq Cw(x,t)+C\qquad\text{for 
		}\;\;(x,t)\in\Omega\times[0,T_{\mathrm{max}}).
	\end{equation*}
Thanks to the strictly positive lower bound given in \eqref{boundw}, we infer that
\begin{equation}
	v(x,t)\leq C(w(x,t)+w(x,t)/w_*)=C(1+1/w_*)w(x,t):=Bw(x,t).
\end{equation}
The second upper bound estimate \eqref{upbv} simply follows from  Lemma \ref{comparisonA} and \eqref{wb0} by taking $\varepsilon=1/(1-\tau\gamma_\infty)$. This completes the proof.
\end{proof}

Next, we prove the H\"older continuity of $w$ and $v$. First,  we recall that the auxiliary function $w$ solves the initial-boundary value problem \begin{subequations}\label{cpv}
	\begin{align}
		\partial_t w + u\gamma(v) & = \mathcal{A}^{-1}[u\gamma(v)]\,, \qquad (t,x)\in (0,T_{\mathrm{max}})\times\Omega\,, \label{cpv1} \\
		- \Delta w + w & = u \,, \qquad (t,x)\in (0,T_{\mathrm{max}})\times\Omega\,, \label{cpv1.5} \\
		\nabla w \cdot \nu & = 0\,, \qquad (t,x)\in (0,T_{\mathrm{max}})\times\partial\Omega\,, \label{cpv2} \\
		w(x,0)  & = w_0(x)\,, \qquad x\in\Omega\,. \label{cpv3}
	\end{align}
\end{subequations}

For any fixed $T\in(0,T_{\mathrm{max}})$, denote $J_T=[0,T]$. Introducing $w^*=\|w_0\|_\infty e^{\gamma^*T}$ being the upper bound of $w$ on $\bar{\Omega}\times J_T$ according to Lemma \ref{keylem}, the upper bound of $v$  on $\bar{\Omega}\times J_T$ then is $v^*=2(w^*+C_\infty)/(1-\tau\gamma_\infty)$ due to \eqref{upbv}. We note that with the upper and lower bounds of $v$, 
under the assumptions of \eqref{g1} and \eqref{gsup}, one 
can find two positive constants $\gamma_*(T),\gamma^*>0$ with $\gamma_*$ depending on $T$ and $\gamma^*$ independent of $T$ such that 
$\gamma(v(x,t))\in[\gamma_*(T),\gamma^*]$ for 
$(x,t)\in\bar{\Omega}\times J_T$.

Now, we may proceed along the lines of \cite[Chapter~V, Section~7]{LSU} to establish the following local energy estimate for $w$; see also \cite[Lemma 4.1]{JLZ22}. 
\begin{lemma}\label{locenergy}  Let $\delta\in (0,1)$. Suppose that $\gamma_*\leq \gamma(v(x,t))\leq \gamma^*$ and $0<w_*\leq w(x,t)\leq w^*\leq$ for $(x,t)\in\bar{\Omega}\times J_T$.  There is $C>0$ depending on $\gamma_*,\gamma^*$ and $w^*$ such that, if $\vartheta\in C^\infty(\bar{\Omega}\times J_T)$, $0\le \vartheta\le 1$, $\sigma\in\{-1,1\}$, and $h\in\mathbb{R}$ are such that
	\begin{equation}
		\sigma w(x,t) - h \le \delta\,, \qquad (x,t)\in \mathrm{supp}\,\vartheta\,, \label{smallc1}
	\end{equation}
	then
	\begin{align}\label{locen}
		& \int_\Omega \vartheta^2 (\sigma w(t) - h)_+^2\ \ dx + \frac{\gamma_*}{2} \int_{t_0}^t \int_\Omega \vartheta^2 |\nabla (\sigma w(s) - h)_+|^2\ \ dx\mathrm{d}s \nonumber\\
		& \qquad  \le \int_\Omega \vartheta^2 (\sigma w(t_0) - h)_+^2\ \ dx + C \int_{t_0}^t \int_\Omega  \left( |\nabla\vartheta|^2 + \vartheta|\partial_t\vartheta| \right) (\sigma w(\tau) - h)_+^2\ \ dx\mathrm{d}s\nonumber \\
		& \qquad\quad + C \int_{t_0}^t  \int_{A_{h,\vartheta,\sigma}(s)} \vartheta\ \ dx \mathrm{d}s 
	\end{align}
	for $0\le t_0\le t\le T$, where 
	\begin{equation*}
		A_{h,\vartheta,\sigma}(s)\triangleq \left\{ x\in \Omega\ :\ \sigma w(x,s) > h \right\}\,, \qquad s\in [0,T]\,.
	\end{equation*}
\end{lemma}
\begin{proof}
	By \eqref{keyid},
	\begin{align}
		\frac{1}{2} \frac{d}{dt} \int_\Omega \vartheta^2 ( \sigma w - h)_+^2\ \ dx &  =  \sigma \int_\Omega \vartheta^2 (\sigma w - h)_+ \partial_t w dx + \int_\Omega  (\sigma w - h)_+^2 \vartheta \partial_t\vartheta dx \nonumber \\
		& = - \sigma \int_\Omega \vartheta^2 (\sigma w - h)_+ u \gamma(v) d x + \sigma \int_\Omega \vartheta^2 (\sigma w - h)_+ \mathcal{A}^{-1}[u\gamma(v)]dx \label{Z1}\\
		& \qquad + \int_\Omega (\sigma w - h)_+^2 \vartheta \partial_t\vartheta d x\,. \nonumber
	\end{align}
By comparison principle of elliptic equations, we observe that $\mathcal{A}^{-1}[u\gamma(v)]\leq \gamma^*w^*$. Hence,
\begin{equation*}
	\sigma \int_\Omega \vartheta^2 (\sigma w - h)_+ \mathcal{A}^{-1}[u\gamma(v)]dx \leq \gamma^*w^*\int_\Omega \vartheta^2 (\sigma w - h)_+dx.
\end{equation*}
	Either $\sigma=1$ and it follows from the boundedness of $\gamma$, \eqref{cpv1.5} and the non-negativity of $u$ and $w$ that
	\begin{align*}
		&- \sigma \int_\Omega \vartheta^2 (\sigma w - h)_+ u \gamma(v) d x \\
		& \le - \gamma_* \int_\Omega \vartheta^2 (w - h)_+ u dx \\
		& = - \gamma_* \int_\Omega \vartheta^2 (w - h)_+ ( w - \Delta w) dx \\
		& \le - \gamma_* \int_\Omega \nabla \left[ \vartheta^2 (w - h)_+ \right]\cdot \nabla w  dx \\
		& \le - \gamma_* \int_\Omega \vartheta^2 |\nabla (w - h)_+ |^2 dx  + 2 \gamma_* \int_\Omega \vartheta |\nabla\vartheta| (w-h)_+ |\nabla (w-h)_+| dx\,.
	\end{align*}
	Or $\sigma=-1$ and we infer  that
	\begin{align*}
		- \sigma \int_\Omega \vartheta^2 (\sigma w - h)_+ u \gamma(v) d x & \le  \gamma^* \int_\Omega \vartheta^2 (-w - h)_+ u\ \ dx \\
		& = \gamma^* \int_\Omega \vartheta^2 (-w - h)_+ (w - \Delta w)\ \ dx \\
		& \le \gamma^* w^* \int_\Omega \vartheta^2 (-w - h)_+\ \ dx \\	
		& \qquad + \gamma^* \int_\Omega \nabla \left[ \vartheta^2 (-w - h)_+ \right]\cdot \nabla w\ \ dx \\
		& \le - \gamma^* \int_\Omega \vartheta^2 |\nabla(-w-h)_+|^2\ \ dx \\
		& \qquad + 2 \gamma^* \int_\Omega \vartheta |\nabla\vartheta| (-w-h)_+ |\nabla (-w-h)_+|dx \\
		& \qquad + C(\gamma^*,w^*) \int_\Omega \vartheta^2 (-w - h)_+ dx\,.
	\end{align*}

%	Recalling that $\gamma^*=\gamma(v_*)\ge \gamma(v^*(T))$ by \eqref{g1}, we have thus shown  that
%	\begin{align*}
%		- \sigma \int_\Omega \vartheta^2 (\sigma w - h)_+ u \gamma(v) d x & \le - \gamma(v^*(T)) \int_\Omega \vartheta^2 |\nabla(\sigma w-h)_+|^2\ \ dx \\
%		& \qquad + C(T) \int_\Omega \vartheta |\nabla\vartheta| (\sigma w-h)_+ |\nabla (\sigma w-h)_+|\ \ dx \\
%		& \qquad + C(T) \int_\Omega \vartheta^2 (\sigma w - h)_+\ \ dx\,.
%	\end{align*}
	Inserting the above estimates in \eqref{Z1} and using  Young's inequality lead us to
	\begin{align*}
		&\frac{1}{2} \frac{d}{dt} \int_\Omega \vartheta^2 ( \sigma w - h)_+^2\ \ dx\\
		 & \le - \gamma_* \int_\Omega \vartheta^2 |\nabla(\sigma w-h)_+|^2\ \ dx \\
		& \qquad + \frac{\gamma_*}{2} \int_\Omega \vartheta^2 |\nabla(\sigma w-h)_+|^2\ \ dx +  C(\gamma_*,\gamma^*) \int_\Omega |\nabla\vartheta|^2 (\sigma w-h)_+^2\ \ dx \\
		& \qquad + C(\gamma^*,w^*) \int_\Omega \vartheta^2 (\sigma w - h)_+\ \ dx   + \int_\Omega (\sigma w - h)_+^2 \vartheta |\partial_t\vartheta| d x \\
		& \le -\frac{\gamma_*}{2} \int_\Omega \vartheta^2 |\nabla(\sigma w-h)_+|^2\ \ dx \\
		& \qquad +  C(\gamma_*,\gamma^*) \int_\Omega \left( |\nabla\vartheta|^2 + \vartheta |\partial_t\vartheta| \right) (\sigma w-h)_+^2\ \ dx + C(\gamma^*,w^*) \int_\Omega \vartheta^2 (\sigma w - h)_+\ \ dx\,.
	\end{align*}
	We now use \eqref{smallc1} to estimate from above the last term by
	\begin{equation*}
		C(\gamma^*,w^*) \int_\Omega \vartheta^2 (\sigma w - h)_+\ \ dx \le C(\gamma^*,w^*)\delta \int_{A_{h,\vartheta,\sigma}} \vartheta^2\ \ dx \le C(\gamma^*,w^*)\int_{A_{h,\vartheta,\sigma}} \vartheta\ \ dx\,,
	\end{equation*}
	and integrate the above differential inequality over $(t_0,t)$ to complete the proof.
\end{proof}
We are now in a position to apply  \cite[Chapter II, Theorem 8.2]{LSU} to obtain a H\"older estimate for $w$, the proof of which was given in \cite[Corollary 4.2]{JLZ22}.
\begin{corollary}\label{holderw}
Under the assumptions of Lemma \ref{locenergy}, there is $\alpha\in(0,1)$ depending on $\gamma_*,\gamma^*,w^*,\delta,N$ and the initial data such that 
$w\in C^{\alpha}(\bar{\Omega}\times J_T)$.
\end{corollary}
Next, we use standard Schauder's estimate to derive  H\"older continuity for $v$.
\begin{lemma}\label{holderv}
	The function $v$ belongs to $C^{\alpha,1+\alpha}(\bar{\Omega}\times J_T)$ with the exponent  $\alpha$ given in Corollary~\ref{holderw}.
\end{lemma}
\begin{proof}
	Set $r \triangleq \mathcal{A}^{-1}[v]$ and  $r_0 \triangleq\mathcal{A}^{-1}[v_0]$. In view of the regularity $v_0\in W^{1,\infty}(\Omega)$, there holds  that $r_0\in C^{2+\alpha}(\bar{\Omega})$ with $\alpha\in(0,1)$ being the exponent given in  Corollary~\ref{holderw}. Besides, we infer from \eqref{cp2} that $r$ is a solution to
\begin{equation*}
	\begin{cases}
		\tau r_t-\Delta r+r=w,\qquad &(x,t)\in 
		\Omega\times(0,T_{\mathrm{max}})\\
		\nabla r\cdot \nu=0,\qquad &(x,t)\in 
		\partial\Omega\times(0,T_{\mathrm{max}})\\
		r(x,0)=r_0(x)\qquad& x\in\Omega,
	\end{cases}
\end{equation*}
	so that Schauder's theory of heat equations, along with Corollary~\ref{holderw}, ensures that $r$ belongs $C^{2+\alpha,1+\alpha}(\bar{\Omega}\times J_T)$. As a result, we obtain that $v =  r-\Delta r\in C^{\alpha,1+\alpha}(\bar{\Omega}\times J_T)$.
\end{proof}
\begin{remark}\label{timedependence}
	 We remark that the exponent $\alpha$ and the H\"older estimates of $v$ obtained above depend on $\gamma_*,\gamma^*,w^*$, $\Omega$ and the initial data. As a result, in this section, they depend on $T$  in view of the dependence of $\gamma_*$ and $w^*$ on time.
\end{remark}

\noindent\textbf{Proof of Theorem \ref{TH0}.} 
With the H\"older continuity of $v$ at hand, we can regard $-\gamma(v)\Delta$ as a generator of evolution operators. Then according to  \cite[Lemma 4.4 and Proposition 4.5]{JLZ22}, we have  for any $q\in(1,\infty)$
\begin{equation}\label{final}
\sup\limits_{0\leq t\leq T}(\|u(\cdot,t)\|_{q}+\|v(\cdot,t)\|_{W^{1,\infty}})\leq C(q,T).
\end{equation}
 Note that the results mentioned above in \cite{JLZ22} hold for the extended system \eqref{cpn}, which covers our case by simply taking $f=n\equiv0$.

Next we write the first equation as
\begin{equation*}
	u_t=\nabla\cdot(\gamma(v(x,t))\nabla u)+\nabla \cdot(u\gamma'(v)\nabla v).
\end{equation*}
Note that $\gamma(v(x,t))\geq\gamma_*$  and  $u\gamma'(v)\nabla v$ is bounded in $L^\infty(0,T; L^q(\Omega))$ with any $q\in(1,\infty)$ by \eqref{final}. This allows us to apply \cite[Lemma A.1]{TaoWin12} by taking $m=1$, fixing any $q_1>N+2$, $q_2>\frac{N+2}{2}$ and choosing $p_0$ sufficiently large. Thus, we finally deduce that
\begin{equation*}
	\sup\limits_{0\leq t\leq T}\|u(\cdot,t)\|_\infty\leq C(T).
\end{equation*}
According to Theorem \ref{local}, we deduce that $T_{\mathrm{max}}=\infty$ and this finishes the proof of Theorem \ref{TH0}.

}
%%%%%%%%%%%%%%%%%%%%%%%%%%%%%%%%%%%%%%%%%%%%
%%%%%%%%%%%%%%%%%%%%%%%%%%%%%%%%%
\section{Uniform-in-time boundedness with non-vanishing motility}
In this section, we study uniform-in-time boundedness of classical solution to 
problem \eqref{cp} with  non-vanishing motility. More 
precisely,    we assume within this section that there are two 
positive constants 
$\gamma_*$ and $\gamma^*$ such that 
\begin{equation}\label{ulbg}
	0<\gamma_*\leq \gamma(\cdot)\leq \gamma^*\;\;\text{on}\;(0,\infty).
\end{equation}
If in particular $\gamma$ is non-increasing on $(0,\infty)$, i.e., 
$\gamma'(\cdot)\leq 0$ on $(0,\infty)$, then
one has $\gamma^*=\gamma(v_*)$.

First, we prove uniform-in-time $L^p$-estimates for $w$ with any $1\leq p<\infty.$
\begin{lemma}\label{Lp}
	Under the assumption of \eqref{g1} and \eqref{ulbg}, for any $1\leq 
	p<\infty$  there holds
	\begin{equation}
		\|w(\cdot,t)\|_p\leq C(p)\,\qquad \forall \;t\in[0,T_{\mathrm{max}}),
	\end{equation}where $C(p)>0$ is a constant depending only on $\Omega$, 
	$\tau$, $\gamma$, $p$ and the initial data.
\end{lemma}
\begin{proof}
	Thanks to the upper and lower bounds of $\gamma$, the non-negativity of $u$, 
	the fact $w-\Delta w=u$, 
	and the comparison principle of elliptic equations, we deduce from the key 
	identity 
	that
	\begin{equation}\label{KeyIneq}
	w_t-\gamma_*\Delta w+\gamma_*w=w_t+\gamma_*u\leq w_t+\gamma(v)u=\A^{-1}[u\gamma(v)]\leq \A^{-1}[\gamma^*u]= \gamma^*w.
	\end{equation}
A multiplication of \eqref{KeyIneq} by $pw^{p-1}$ with any $p\geq2$ and an 
integration over $\Omega$ give rise to
\begin{equation}
		\frac{d}{dt}\int_\Omega w^pdx+\frac{p-1}{4p}\gamma_*\int_\Omega |\nabla w^{p/2}
	|^2dx+p\gamma_*\int_\Omega w^p dx\leq 
	p\gamma^*\int_\Omega w^{p}dx.
\end{equation}
Recall the  Gagliardo-Nirenberg inequality
\begin{equation}
	\|z\|\leq C\|\nabla z\|^{\alpha}\|z\|^{1-\alpha}_{1}+C\|z\|_{1}
\end{equation}
with
\begin{equation}
	\alpha=\frac{N}{N+2}\in(0,1).
\end{equation}
Letting $z=w^{p/2}$, it follows by the above inequality and Young's inequality 
that
\begin{align*}
	p\gamma^*\int_\Omega w^pdx\leq& Cp\gamma^*\bigg(\int_\Omega|\nabla w^{p/2}|^2dx\bigg)^{\alpha}\bigg(\int_\Omega w^{p/2}\bigg)^{2(1-\alpha)}+Cp\gamma^*\bigg(\int_\Omega w^{p/2}dx\bigg)^2\\
	\leq & \frac{\gamma_*}{16} \int_\Omega|\nabla w^{p/2}|^2dx+C(N,p,\gamma_*,\gamma^*)\bigg(\int_\Omega w^{p/2}dx\bigg)^2.
\end{align*} Since $p\geq2$, we notice that $\frac{p-1}{4p}\geq\frac18$.
Thus, we arrive at
\begin{equation}
		\frac{d}{dt}\int_\Omega w^pdx+\frac{\gamma_*}{16}\int_\Omega |\nabla w^{p/2}
	|^2dx+p\gamma_*\int_\Omega w^p dx\leq  C(N,p,\gamma_*,\gamma^*)\bigg(\int_\Omega w^{p/2}dx\bigg)^2.
\end{equation}
Solving the above ODI with $p=p_k=2^k$, 
$(k=1,2,3,...)$ successively in view of the fact $\int_\Omega wdx=\int_\Omega 
udx=\|u_0\|_1$ yields to the $L^p$-estimates for $w$ with any $p\geq2$ and the 
remainder 
case $1\leq p<2$ follows by the obvious embedding $L^2\hookrightarrow L^p$ on bounded domains. This 
completes the proof.
\end{proof}
Now we are ready to prove the uniform-in-time bound of $w$. 
\begin{lemma}\label{wbound}
	Under the assumptions of Theorem \ref{TH1}, there is a positive constant 
	$C$ depending only on $\Omega$, $\tau$, $\gamma$  and the initial data such 
	that
	\begin{equation}
	\|w(\cdot,t)\|_\infty\leq C,\qquad\forall\,t\in[0,T_{\mathrm{max}}).
	\end{equation}
\end{lemma}
\begin{proof}
The proof is based on a simple comparison argument. Let $z$ be 
the solution of the following linear equation:
\begin{equation}
	\begin{cases}
		z_t-\gamma_*\Delta z+\gamma_* z=\gamma^*w \qquad &(x,t)\in 
		\Omega\times (0,T_{\mathrm{max}})\\
		\nabla z\cdot \nu=0\qquad &(x,t)\in 
		\partial\Omega\times(0,T_{\mathrm{max}}) \\
		z(x,0)=w_0(x)\qquad &x\in\Omega.
	\end{cases}
\end{equation}
Then we have the following representation formula
\begin{equation*}
	z(\cdot, t)=e^{\gamma_*(\Delta-1)t}w_0+\gamma^*\int_0^te^{\gamma_*(\Delta-1)(t-s)}
	w(\cdot, s)ds.
\end{equation*}
It follows from above and Lemma \ref{Lp} that
\begin{align*}
	\|z(\cdot,t)\|_\infty\leq&\|e^{\gamma_*(\Delta-1)t}w_0\|_\infty+\gamma^*\int_0^t
	\|e^{\gamma_*(\Delta-1)(t-s)}w(s)\|_\infty ds\\
	\leq& C\|w_0\|_\infty+C\gamma^*\int_0^t e^{-\gamma_*(t-s)}
	(t-s)^{-\frac{N}{2p}}\|w(s)\|_pds\\
	\leq &C(\|w_0\|_\infty,\gamma_*,\gamma^*,p,N,\Omega)
\end{align*}provided that $p>N/2$, since
\begin{equation*}
	\int_0^\infty e^{-\gamma_*s}s^{-\frac{N}{2p}}ds<\infty.
\end{equation*}
Then by comparison principle of heat 
equations, we have $0\le w\leq z\leq C(\|w_0\|_\infty,\gamma_*,\gamma^*,N,\Omega)$ in 
view of \eqref{KeyIneq}. This 
completes 
the proof.
\end{proof}

%%%%%%%%%%%%%%%%%%%%%%%%%%%%%%%%%%%%%%%%
%%%%%%%%%%%%%%%%%%%%%%%%%%%%%%%%%%%%%%%
\noindent\textbf{Proof of Theorem \ref{TH1}.}  {Fix any $T\in(0, T_{\mathrm{max}})$. Although we do not have the upper bound of $v$ at the present stage, by Lemma \ref{locenergy}, the strictly positive lower and upper bounds assumption \eqref{ulbg} together with the uniform-in-time upper bound for $w$ still allows us to establish the  local energy estimate \eqref{locen} for $w$, where the constant $C>0$ depend on $\gamma_*,\gamma^*$ and $\sup\limits_{t\in[0,T_{\mathrm{max}})}\|w\|_\infty$, and thus are both independent of $T$. As a result, according to Corollary \ref{holderw} and Lemma \ref{holderv}, there is $\alpha>0$ independent of $T>0$ such that $w$ and $v$ belong to H\"older spaces $C^\alpha(\bar{\Omega}\times J_T)$ and $C^{\alpha,1+\alpha}(\bar{\Omega}\times J_T)$, respectively. Moreover, their H\"older estimates are independent of $T$ as well. 
	
	Finally, by the results in \cite[Lemma 6.7 and Proposition 6.8]{JLZ22}, there is $C(q)>0$ independent of $T$ and  $T_{\mathrm{max}}$ such that $\sup\limits_{0\leq t<T_{\mathrm{max}}}(\|u(\cdot,t)\|_{q}+\|v(\cdot,t)\|_{W^{1,\infty}})\leq C(q)$ for any $q\in(1,\infty)$. Then it follows by \cite[Lemma A.1]{TaoWin12} that $\sup\limits_{0\leq t<T_{\mathrm{max}}}\|u(\cdot,t)\|_{\infty}<\infty.$   This completes the proof according to Theorem \ref{local}. \qed

}

%%%%%%%%%%%%%%%%%%%%%%%%%%%%%%%%%%%%%%%%%%%%%%%%%%%%

%%%%%%%%%%%%%%%%%%%%%%%%%%%%%%%%%%%%%%%%%%%%%%%%%%
%%%%%%%%%%%%%%%%%%%%%%%%%%%%%%%%%%%%%%%%%%%%%%%%%%

\section{Uniform-in-time boundedness with decaying motility}	
In this section, we consider problem \eqref{cp} with a decaying motility function that may vanish at infinity. Under the circumstances, uniform boundedness of classical solution is related to the decay rate of $\gamma$ at infinity.
\subsection{The case $N=2$}
In this part, we first consider the two-dimensional case and we  assume that 
$\gamma$ satisfies \eqref{g1}, \eqref{gsup} and \eqref{growth2da} with some 
$\chi>0$. According to Theorem \ref{TH0}, classical solution exists globally. One also observes due to \eqref{g1} and \eqref{gsup} that $\gamma$ has an 
upper bound on 
$[v_*,\infty)$ denoted by $\gamma^*$. Then the proof for Theorem \ref{TH2a} can 
be 
carried out in a similar manner as done in \cite[Section 3.1]{FJ20a}. We report it here for reader's convenience.
{
%%%%%%%%%%%%%%%%
\begin{lemma}\label{wH1a} Assume \eqref{g1}, \eqref{gsup} and \eqref{growth2da} with some 
	$\chi>0$. Suppose that $\|u_0\|_1<4\pi(1-\tau\gamma_\infty)/\chi$. There is $C(\chi)>0$ depends only on $\Omega,\gamma,\tau$ and the 
	initial data such that
\begin{equation}
	\sup\limits_{t>0}\bigg(\|\nabla w\|^2+\|w\|^2+\int_t^{t+1}\int_\Omega 
	u^2\gamma(v)dxds\bigg)\leq C(\chi).
\end{equation}
\end{lemma}
%%%%%%%%%%%%%%%%
\begin{proof}
	We multiply the key identity \eqref{keyid} by $u$ and integrate over $\Omega$. Recalling that $w=\mathcal{A}^{-1}[u]$ and that $\mathcal{A}^{-1}$ is self-adjoint, we obtain 
	\begin{align*}
		\frac{1}{2}\frac{d}{dt}\left(\|\nabla w\|_2^2 + \|w\|_2^2\right)+\int_\Omega u^2\gamma(v)\,d x=&\int_\Omega u \mathcal{A}^{-1}[ u\gamma(v)]\,d x\\
		=&\int_\Omega  u\gamma(v) \mathcal{A}^{-1}[u]\,d x\\
		=&\int_\Omega \gamma(v) uw\, d x\,.
	\end{align*}
	Hence, by the upper bound $\gamma^*$ for $\gamma$,
	\begin{equation*}
		\frac{1}{2}\frac{d}{dt}\left(\|\nabla w\|_2^2 + \|w\|_2^2\right)+\int_\Omega u^2\gamma(v)\,d x \le \gamma^*\int_\Omega uw\, d x\,.
	\end{equation*}
	Also, note that
	\begin{equation*}
		\|\nabla w\|_2^2+\|w\|_2^2=\int_\Omega w u\, d x.
	\end{equation*} 
	Combining the above inequalities and using Young's inequality, we arrive at
	\begin{align*}
		\frac{d}{dt}\left(\|\nabla w\|_2^2 + \|w\|_2^2\right) & +\|\nabla w\|_2^2 +  \|w\|_2^2 +2\int_\Omega u^2 \gamma(v)\,d x\\
		=&(2\gamma^*  + 1) \int_\Omega w u \,d x\\
		\leq&\int_\Omega u^2 \gamma(v)\, d x
		+\frac{(2\gamma^*+1)^2}{4}\int_\Omega \frac{w^2}{\gamma(v)}\,d x.
	\end{align*} 
	We thus  obtain
	\begin{equation}\label{wb00}
		\frac{d}{dt}\left(\|\nabla w\|_2^2 +  \|w\|_2^2\right) + \|\nabla w\|_2^2 +  \|w\|_2^2	+\int_\Omega u^2 \gamma(v)\, d x 
		\leq C\int_\Omega \frac{w^2}{\gamma(v)}\, d x\,
	\end{equation} with $C>0$ independent of time.
	
	We now deduce from the assumption~\eqref{growth2da} that there exist $b>0$  and $s_\chi>v_*$ depending on $\chi$ such that  $e^{\chi s} \gamma(s) \ge 1/b$ for all $s\geq s_\chi$. Besides, owing to the continuity of $\gamma$, there holds $\gamma(s)\ge \gamma(s_\chi,v_*)$ for $s\in [v_*,s_\chi]$, we end up with
	\begin{equation}\label{cond_gamma0}
		\frac{1}{\gamma(s)}\leq \max\left\{ b , \frac{1}{\gamma(s_\chi,v_*)} \right\} e^{\chi s}\,, \qquad s\geq v_*\,.
	\end{equation}
	Now, for $\varepsilon>0$, we infer from  \eqref{cond_gamma0},  \eqref{eupctrol} and the elementary inequality $e^{\varepsilon s} \ge \varepsilon^2 s^2$, $s>0$,  that
	\begin{equation}
		\begin{split}
			\int_\Omega \frac{w^2}{\gamma(v)}\, d x & \le C(\chi) \int_\Omega w^2 e^{\chi v}\, d x \le C(\chi) \int_\Omega w^2 e^{ \chi(\frac{1}{1-\tau\gamma_\infty}+\varepsilon) (w+C_\infty)}\, d x \\
			& \le \frac{C(\chi)e^{\chi C_\infty(\frac{1}{1-\tau\gamma_\infty}+\varepsilon)}}{\chi^2 \varepsilon^2} \int_\Omega e^{\chi(\frac{1}{1-\tau\gamma_\infty}+2\varepsilon)w}\, d x \,.
		\end{split} \label{unic0}
	\end{equation}
	Since $\|u\|_1=\|u_0\|_1 < 4\pi(1-\tau\gamma_\infty)/\chi$ by \eqref{e0}, we may choose $\varepsilon_\chi>0$ such that
	\begin{equation}
		\chi (\frac{1}{1-\tau\gamma_\infty}+ 2 \varepsilon_\chi) < \frac{4\pi}{\|u_0\|_1} \qquad \left( \text{ say }\;\;  \varepsilon_\chi \triangleq \frac{\pi}{\chi \|u_0\|_1} - \frac{1}{4(1-\tau\gamma_\infty)} \right) \label{y13}
	\end{equation}
	and deduce from  Lemma~\ref{lm2e} that
	\begin{equation}
		\int_\Omega e^{\chi(\frac{1}{1-\tau\gamma_\infty}+2\varepsilon_\chi)w}\, d x \le C(\|u_0\|_1,\chi(\frac{1}{1-\tau\gamma_\infty}+2\varepsilon_\chi))\,. \label{y12}
	\end{equation}
	Gathering \eqref{wb00}, \eqref{unic0} (with $\varepsilon=\varepsilon_\chi$), and \eqref{y12} gives
	\begin{equation*}
		\frac{d}{dt}\left(\|\nabla w\|_2^2 +  \|w\|_2^2\right) + \|\nabla w\|_2^2 +  \|w\|_2^2	+\int_\Omega u^2 \gamma(v)\, d x 
		\leq C(\chi)\,,
	\end{equation*} 
	from which Lemma~\ref{wH1a} follows after integration with respect to time.
\end{proof}
%%%%%%%%%%%%%%%%%%%%%%%%%%%%%%%%%%%%%%%%%%%%

\begin{proposition} \label{cst10}
Assume \eqref{g1}, \eqref{gsup} and \eqref{growth2da} with some 
$\chi>0$ and that $\|u_0\|_1<4\pi(1-\tau\gamma_\infty)/\chi$.	There is $C(\chi)>0$ such that
	\begin{equation}
		\sup\limits_{t\geq0}\left(\|w(t)\|_{\infty}+\|v(t)\|_{\infty}\right)\leq C(\chi)\,.
	\end{equation}
\end{proposition}
%%%%%%%%%%%%%%%%

\begin{proof} Let $p\in (1,2)$ and $\varepsilon>0$ to be specified later. Since $u= \mathcal{A}[w]$, we infer from the Sobolev embedding theorem, H\"{o}lder's inequality,  \eqref{cond_gamma0} and \eqref{eupctrol} that
	\begin{align}\label{winfty}
		\|w\|_{\infty} \leq &C(p) \|w\|_{W^{2,p}} \le C(p)\|u\|_{p} \nonumber\\
		\leq & C(p)\left(\int_\Omega u^2\gamma(v)\,d x\right)^{\frac{1}{2}} \left(\int_\Omega (\gamma(v))^{-\frac{p}{2-p}}\,d x\right)^{\frac{2-p}{2p}}\nonumber\\
		\leq&C(p,\chi) \left(\int_\Omega u^2\gamma(v)\,d x\right)^{\frac{1}{2}} \left(\int_\Omega e^{\frac{\chi p}{2-p}v}\,d x\right)^{\frac{2-p}{2p}} \nonumber\\
		\leq& C(p,\chi,\varepsilon) \left(\int_\Omega u^2\gamma(v)\,d x\right)^{\frac{1}{2}} \left(\int_\Omega e^{\frac{\chi p}{2-p}(\frac{1}{1-\tau\gamma_\infty}+\varepsilon)w}\,d x\right)^{\frac{2-p}{2p}}\,.
	\end{align}
	Choosing $p=p_\chi \triangleq \frac{2+4\varepsilon_\chi(1-\tau\gamma_\infty)}{2+3\varepsilon_\chi(1-\tau\gamma_\infty)}\in (1,2)$ with $\varepsilon_\chi$ defined in \eqref{y13}, we observe that
	\begin{equation*}
		\frac{\chi p_\chi}{(2-p_\chi)} (\frac{1}{1-\tau\gamma_\infty}+\varepsilon_\chi) = \chi (\frac{1}{1-\tau\gamma_\infty}+2\varepsilon_\chi)
	\end{equation*}
	and we deduce from \eqref{y12} and \eqref{winfty} (with $\varepsilon=\varepsilon_\chi$ and $p=p_\chi$) that
	\begin{equation*}
		\|w\|_{\infty}\leq C(\chi) \left(\int_\Omega u^2\gamma(v)\,d x\right)^{\frac{1}{2}}.
	\end{equation*}	
	Combining the above estimate with Lemma~\ref{wH1a}, we obtain, for $t\geq 0$,
	\begin{equation}
		\int_t^{t+1}\|w(\cdot,s)\|_{\infty}\,d s\leq \left( \int_t^{t+1} \|w(\cdot,s)\|_\infty^2\,d s \right)^{\frac{1}{2}} \le C(\chi) \left( \int_t^{t+1}\int_\Omega u^2\gamma(v)\,d x d s \right)^{\frac{1}{2}}\leq C(\chi)\,. \label{y14}
	\end{equation}
	Now, observing that the key identity \eqref{keyid}, the non-negativity of $u$ and $\gamma$, and the elliptic comparison principle imply that
	\begin{equation*}
		\partial_t w \le \partial_t w + u\gamma(v) =  \mathcal{A}^{-1}\left[  u\gamma(v)  \right] \le \gamma^* \mathcal{A}^{-1}[u] = \gamma^*  w\,,
	\end{equation*}
	we realize that $w(x,t+1) \le e^{\gamma^* (t+1-s)} w(x,s)$ for all $(x,s)\in\Omega\times(t,t+1)$. Consequently,
	\begin{equation*}
		\|w(\cdot,t+1)\|_\infty e^{\gamma^* (s-t-1)} \le \|w(\cdot,s)\|_\infty\,, \qquad s\in [t,t+1]\,,
	\end{equation*}
	and it follows from \eqref{y14} and the above inequality after integration with respect to $s$ over $[t,t+1]$ that 
	\begin{equation*}
		\|w(\cdot,t+1)\|_\infty \frac{1 - e^{- \gamma^* }}{\gamma^*} \le C(\chi)\,.
	\end{equation*}
	Since we already know that $\|w(\cdot,t)\|_\infty \le C$ for $t\in [0,1]$ by Lemma~\ref{keylem}, we have thus proved that
	\begin{equation*}
		\|w(\cdot,t)\|_\infty \leq C(\chi) \qquad\text{for}\;\;t\geq 0\,.
	\end{equation*} 
	The claimed boundedness of $v$ now readily follows in view of \eqref{upcontrol} and the non-negativity of $v$.
\end{proof}

%%%%%%%%%%%%%%%%%%%%
\begin{proof}[Proof of Theorem~\ref{TH2a}]
	With Proposition~\ref{cst10} at hand, $\gamma$ has  strictly positive time-independent upper and lower bounds on $\bar{\Omega}\times[0,\infty)$,  the result then directly follows from Theorem~\ref{TH1}. 
\end{proof}

}

\subsection{The case $N\geq3$}
In this part, we assume that $N\geq3$ and $\gamma$  satisfies \eqref{g1}, 
\eqref{gsup}, \eqref{gamma2} and \eqref{gamm3}.
We show that under the above assumptions, the technical condition \eqref{A5} and the non-increasing property of $\gamma$  in \cite[Theorem1.4]{JLZ22} are both removable.

Note that if $\gamma$ satisfies \eqref{gamma2} with $k= l=0$, there are time-independent positive constants $\gamma_*$ and $\gamma^*$ such that $\gamma_*\leq \gamma(s)\leq \gamma^*$ for $s\in[v_*,\infty)$ according to Lemma 
\ref{gamma2b}, then the assertion of Theorem \ref{TH2} follows by Theorem 
\ref{TH1} in this special case. In the sequel, we aim to provide a unified approach  to prove Theorem \ref{TH2} for all cases under the assumption \eqref{gamm3}.

To begin with, we need to establish the following reverse control.
\begin{lemma}\label{revlem0}
	Suppose that $\gamma$ satisfies \eqref{g1} and there is $C_l>0$ such that $\gamma(s)\leq C_l s^{-l}$ for $s\geq v_*$ with   some $l>0$. Then there is a positive time-independent constant $A>0$ such that
	\begin{equation}\label{rev}
		Aw(x,t)\leq v(x,t),\qquad \text{for }\;(x,t)\in\Omega\times[0,T_{\mathrm{max}}).
	\end{equation}
\end{lemma}
\begin{proof}
	We recall that  by the key identity \eqref{keyid}
	\begin{align}\label{rev000}
		\L[v]=u=&w-\Delta w=\L[w]-\tau w_t\non\\
		=&\L[w]+\tau u\gamma(v)-\tau \A^{-1}[u\gamma(v)].
	\end{align}
	Since $\gamma(s)\leq C_ls^{-l}$ for $s\geq v_*$ with some $l>0$, we infer that
	\begin{align}\label{comrev0}
		u\gamma(v)\leq& C_lu v^{-l}\nonumber\\
		=&C_l(\tau v_t-\Delta v+v)v^{-l}\nonumber\\
		=&C_l\bigg(\tau \partial_t\Gamma(v)-\Delta \Gamma(v)+\Gamma(v)-lv^{-l-1}|\nabla v|^2+v^{1-l}-\Gamma(v)\bigg)\nonumber\\
		=&C_l\bigg(\mathcal{L}[\Gamma(v)]-lv^{-l-1}|\nabla v|^2+v^{1-l}-\Gamma(v)\bigg),
	\end{align}where 
	\begin{equation}\label{defGamma}
		\Gamma(s)\triangleq\int_1^s \eta^{-l}d\eta=\begin{cases}
			\frac{s^{1-l}-1}{1-l},\qquad&\text{if}\;l\neq1,\\
			\log s,\qquad& \text{if}\;l=1.
		\end{cases}
	\end{equation}Denote $s_+=\max\{s,0\}$.
	We observe that for any $s\geq a>0$ ,
	\begin{equation}\label{Gap1}
		s^{1-l}-\Gamma(s)=\begin{cases}\frac{1-ls^{1-l}}{1-l}\qquad&\text{if}\;l\neq1,\\
			1-\log s\leq (1-\log a)_+\qquad& \text{if}\;l=1.
		\end{cases}
	\end{equation}Let
	\begin{equation}\label{Gap2}
		C(a,l):=\begin{cases}
			\frac{la^{1-l}}{l-1}\qquad&\text{if}\;l>1,\\
			\frac{1}{1-l}\qquad&\text{if}\;l<1,\\
			(1-\log a)_+\qquad& \text{if}\;l=1.
		\end{cases}
	\end{equation} 
Then, there is $C'=C(v_*,l)\geq0$ such that
\begin{equation}\label{Gap3}
	v^{1-l}-\Gamma(v)\leq C' \;\;\text{for} \quad(x,t)\in\Omega\times[0,T_{\mathrm{max}}),
\end{equation}
which together with the non-negativity of  $lv^{-l-1}|\nabla v|^2\geq0$, and \eqref{comrev0} implies that 
	\begin{align}
		u\gamma(v)\leq C_l\L[\Gamma(v)+C']
	\end{align}and hence by comparison principle that
	\begin{equation*}
		\A^{-1}[u\gamma(v)]\leq 
		C_l\A^{-1}[\L[\Gamma(v)+C']]=C_l\L[\A^{-1}[\Gamma(v)+C']].
	\end{equation*}
	It follows from \eqref{rev000}, the non-negativity of $u\gamma(v)$ and the above that
	\begin{align}
		\L[v]\geq \L[w]-C_l\tau\L[\A^{-1}[\Gamma(v)+C']],
	\end{align}which according to comparison principle of heat equations, gives rise to
	\begin{equation}\label{rev001}
		w\leq 
		v+C_l\tau \A^{-1}[\Gamma(v)]+C_0,\qquad\text{for}\;(x,t)\in\Omega\times[0,T_{\mathrm{max}})
	\end{equation}with a generic positive constant $C_0\geq C'$ ensuring that 
	\begin{equation*}
		w_0\leq v_0+C_l\tau\A^{-1}[\Gamma(v_0)]+C_0\qquad\text{in}\;\bar{\Omega}.
	\end{equation*}
	
	Next,  noticing that assumption  \eqref{gsup} is satisfied since $\gamma(s)\leq C_l s^{-l}$ with $l>0$, by Lemma \ref{keylem},  $v\leq Bw$ 
	for $(x,t)\in\Omega\times[0,T_{\mathrm{max}})$.	Since $\Gamma(s)$ is monotone increasing, we infer  that
	\begin{equation*}
		\Gamma(v)\leq \Gamma(Bw)=\begin{cases}
			\frac{B^{1-l}w^{1-l}-1}{1-l}=B^{1-l}\Gamma(w)+\frac{B^{1-l}-1}{1-l}\qquad&\text{if}\;l\neq1,\\
			\log B+\log w\qquad&\text{if}\;l=1.
		\end{cases}
	\end{equation*}
	We note that for $l>1$, there holds $\Gamma(Bw)\leq 
	\frac{1}{l-1}$. Then it 
	follows from \eqref{rev001} and comparison principle of elliptic equations that 
	when $l>1$,
	\begin{equation*}
		w\leq v+C_l\tau\A^{-1}[\Gamma(Bw)]+C_0\leq v+C_0'
	\end{equation*}with $C_0'=\frac{C_l\tau}{l-1}+C_0$.
	
	It remains to consider the case $0<l\leq 1$. 
	%By comparison principle,
	%\begin{equation}\label{rev002}
	%\A^{-1}[\Gamma(v)]\leq B^{1-l}\A^{-1}[\Gamma(w)].
	%\end{equation}
	Recall that $w-\Delta w=u$ and observe that 
	\begin{align}\label{rev002}
		u w^{-l}=&\Gamma(w)-\Delta \Gamma(w)-lw^{-l-1}|\nabla w|^2+(w^{1-l}-\Gamma(w))\nonumber\\
		=&\A[\Gamma(w)]-lw^{-l-1}|\nabla w|^2+(w^{1-l}-\Gamma(w)).
	\end{align}
	When $l\in(0,1)$, there holds that
	\begin{equation*}
		w^{1-l}-\Gamma(w)=\frac{1-lw^{1-l}}{1-l}=1-l\Gamma(w).
	\end{equation*}
	Therefore, we infer from the above identity, \eqref{rev002}, the non-negativity of $uw^{-l}$ and 
	$lw^{-l-1}|\nabla w|^2$ that
	\begin{equation}
		l\Gamma(w)\leq \A[\Gamma(w)]+1.
	\end{equation}
	Hence, it follows by comparison principle of elliptic 
	equations that
	\begin{align}\label{rev003}
		\A^{-1}[\Gamma(w)]\leq \frac{1}{l}(\Gamma(w)+1)
	\end{align}
	together with \eqref{rev001} and Young's inequality yielding that
	\begin{align*}
		w\leq&  v+C_l\tau\A^{-1}[\Gamma(Bw)]+C_0\\
		=& v+C_l\tau\A^{-1}[B^{1-l}\Gamma(w)+\frac{B^{1-l}-1}{1-l}]+C_0\\
		\leq& 
		v+\frac{C_l\tau B^{1-l}}{l}\Gamma(w)+\frac{C_l\tau B^{1-l}}{l}+\frac{C_l\tau B^{1-l}}{1-l}+C_0\\
		= &v+\frac{C_l\tau B^{1-l}w^{1-l}}{l(1-l)}+C_0\\
		\leq &v+\frac12 w+C_0'.
	\end{align*}
	Thus, for $l\in(0,1)$ we deduce   that
	\begin{equation*}
		w\leq 2(v+C_0').
	\end{equation*}
	Last, when $l=1$,  we recall that 
	$\Gamma(w)=\log w$ and $w^{1-l}-\Gamma(w)=1-\log w$. We further infer 
	from 
	\eqref{rev002} that
	\begin{equation*}
		\Gamma(w)\leq \A[\Gamma(w)]+1.
	\end{equation*}
	Thus,
	\begin{equation*}
		\A^{-1}[\Gamma(w)]\leq\Gamma(w)+1
	\end{equation*}due to an application of the comparison principle. Finally, we 
	deduce from \eqref{rev001} that
	\begin{align}
		w\leq&v+C_l\tau\A^{-1}[\Gamma(Bw)]+C_0\nonumber\\
		\leq &v+C_l\tau\A^{-1}[\log B+\Gamma(w)]+C_0\nonumber\\
		\leq & v+C_l\tau\log B+C_l\tau\Gamma(w)+C_l\tau+C_0\nonumber\\
		=& v+C_l\tau\log B+C_l\tau\log w+C_l\tau+C_0.
	\end{align}
	Since $\log w\leq w^{\delta}$ for any $\delta>0$, we infer by Young's 
	inequality again that
	\begin{equation*}
		\frac12 w\leq v+C_0''.
	\end{equation*}
	In summary, we establish that
	\begin{equation*}
		w\leq C(v+1)=C(v+v/v_*)=C(1+1/v_*)v:=\frac{1}{A}v.
	\end{equation*}
	This completes the proof.
\end{proof}
In view of Lemma \ref{keylem} and Lemma \ref{revlem0}, we have the following two-sided control.
\begin{corollary}\label{revlem}
	Suppose that $\gamma$ satisfies \eqref{gamma2} with some $k\geq 
	l>0$.Then there are positive time-independent constants $A,B>0$ such that
	\begin{equation}\label{dualcontrol}
		Aw(x,t)\leq v(x,t)\leq Bw(x,t),\qquad \text{for }\;(x,t)\in\Omega\times[0,T_{\mathrm{max}}).
	\end{equation}
\end{corollary}
Now, we prove the following key inequality for $w$ when $\gamma$ satisfies \eqref{gamma2}.
\begin{lemma}\label{lemineq}
	Suppose that $\gamma$ satisfies \eqref{g1}, \eqref{gsup} and \eqref{gamma2}, there exist time-independent positive constants $C_1, C_2$ and $C_3$, such that for $(x,t)\in\Omega\times[0,T_{\mathrm{max}})$
	\begin{equation}\label{keyineq}
		w_t+C_1B^{-k}w^{-k}u\leq C_2A^{-l}(\Gamma(w)+C_3)
	\end{equation}where $\Gamma(\cdot)$ is defined by \eqref{defGamma}.
\end{lemma}
\begin{proof} 
Firstly, we consider the case $k=l=0$. According to Lemma \ref{gamma2b}, there are two time-independent positive constants $C_1, C_2$ such that
\begin{equation*}
	0<C_1\leq \gamma(s)\leq C_2\qquad\text{for}\;s\in[v_*,\infty).
\end{equation*} It follows from the key identity, the non-negativity of $u$, and the comparison principle that
\begin{equation*}
w_t+C_1u\leq	w_t+u\gamma(v)=\mathcal{A}^{-1}[u\gamma(v)]\leq C_2\mathcal{A}^{-1}[u]=C_2w=C_2\Gamma(w)+C_2.
\end{equation*}

 Secondly, 	we consider the case  $l=0$ and $k>0$.  By Lemma \ref{gamma2b} again, there are two generic positive constants $C_1,C_2>0$ depending on $\Omega$, $\gamma$ and the initial data such that
	\begin{equation*}
		C_1 s^{-k}\leq \gamma(s)\leq C_2\qquad\text{for}\;s\in[v_*,\infty).
	\end{equation*}
	Since $\gamma$ also satisfies \eqref{gsup} (note that \eqref{gsup} is only  explicitly needed here for Theorem \ref{TH2}), we infer by the upper control $v\leq Bw$ that 
	\begin{equation*}
		C_1B^{-k}w^{-k}\leq C_1v^{-k}\leq\gamma(v(x,t))\leq C_2,\qquad \text{for}\;(x,t)\in\Omega\times[0,T_{\mathrm{max}}),
	\end{equation*}
	which being substituted into the key identity \eqref{keyid} yields that
	\begin{equation*}
		w_t+C_1B^{-k}w^{-k}u\leq w_t+\gamma(v)u\leq C_2\A^{-1}[u]=C_2w=C_2\Gamma(w)+C_2.
	\end{equation*}	

Lastly, we consider the case $k\geq l>0$. Thanks to Lemma \ref{gamma2b} again, there holds
\begin{equation}\label{case3}
	C_1 s^{-k}\leq \gamma(s)\leq C_2 s^{-l}\qquad\text{for}\;s\in[v_*,\infty),
\end{equation} which together with the the two-sided control in Corollary \ref{revlem} yields that
\begin{equation}
	C_1B^{-k}w^{-k}\leq C_1v^{-k}\leq \gamma(v)\leq C_2v^{-l}\leq C_2A^{-l}w^{-l}
\end{equation}Then an application of the  comparison principle to the key 
identity gives rise to
\begin{equation}
w_t+C_1B^{-k}w^{-k}u\leq w_t+u\gamma(v)=\mathcal{A}^{-1}[u\gamma(v)]\leq C_2A^{-l}\mathcal{A}^{-1}[uw^{-l}].
\end{equation}
Finally, by \eqref{Gap1} and \eqref{Gap2} there is $C''=C(w_*,l)\geq0$ such that
\begin{align*}
w^{1-l}-\Gamma(w)\leq C''.
\end{align*}
Thus, it follows from \eqref{rev002} that
\begin{equation}
	\mathcal{A}^{-1}[uw^{-l}]\leq \Gamma(w)+C''.
\end{equation}
This completes the proof.
\end{proof} 
{
Next, we establish the following energy inequality.
\begin{lemma}\label{lmp} There are time-independent positive constants $\lambda_0>0$ and $C_0>0$ such that, for any $p>1+k$,
	\begin{equation*}
		\frac{d}{dt}\|w\|_p^{p}+\frac{\lambda_0 p(p-k-1)}{(p-k)^2} \|\nabla w^{\frac{p-k}{2}}\|_2^2+\lambda_0 p\| w\|_{p-k}^{p-k}\leq  C_0 p\int_\Omega\left( w^{p-1}\Gamma(w)+ w^{p-1}\right) d x\,.
	\end{equation*}
\end{lemma}
\begin{proof}
	Multiplying the inequality \eqref{keyineq} by $w^{p-1}$ for some $p>1+k$, we obtain that
	\begin{equation*}
		\frac{1}{p}\frac{d}{dt}\|w\|_p^p+C_1B^{-k}\int_\Omega uw^{p-k-1}\ d x \leq C_2A^{-l}\int_\Omega w^{p-1}(\Gamma(w)+C_3) d x.
	\end{equation*}	
Recalling that $w-\Delta w=u$, we observe that
	\begin{align*}
		\int_\Omega w^{p-k-1}u\ dx & = \int_\Omega w^{p-k-1}(w-\Delta w)\ d x = \|w\|_{p-k}^{p-k} + (p-k-1) \int_\Omega w^{p-k-2} |\nabla w|^2\ dx \\
		& =  \|w\|_{p-k}^{p-k} + \frac{4(p-k-1)}{(p-k)^2} \|\nabla w^{\frac{p-k}{2}}\|_2^2.
	\end{align*}
The assertion follows by collecting the above estimates and this completes the proof.
\end{proof}

\noindent\textbf{Proof of Theorem \ref{TH2}}.
We point out that energy inequality established for $w$ in Lemma \ref{lmp} is exactly the same as the one in \cite[Lemma 4.3]{JL21} for $v$ (see also \cite[Lemma 6.3]{JLZ22} for the same energy inequality established for an auxiliary function $\tilde{S}$). Thus, starting from this energy inequality, we can proceed in the same manner by delicate Moser-Alikakos iteration argument to prove that under the same assumption \eqref{gamm3} as done in \cite{JL21,JLZ22}, there is a time-independent constant $C>0$  such that
\begin{equation}
\sup\limits_{0\leq t<T_{\mathrm{max}}}\|w(\cdot,t)\|_\infty\leq C.
\end{equation}
The uniform-in-time $L^\infty$-bound for $v$ follows as well by Corollary \ref{cor1}.  As a result, $\gamma$ is bounded from above and below by time-independent positive constants on $\Omega\times[0,T_{\mathrm{max}})$ and we may apply Theorem \ref{TH1} to conclude Theorem \ref{TH2a} now.
\qed
}
\bigskip

\noindent\textbf{Acknowledgments.} 
This work is supported by National Natural Science Foundation of China (NSFC)
under grants No. 12271505 and No. 12071084, and by the Training Program of Interdisciplinary Cooperation of Innovation Academy for Precision Measurement Science and Technology, CAS (No. S21S3202). 

Jiang is grateful to 
Prof. Philippe Lauren\c cot for 
helpful discussions and insightful comments which improve the proof. Both 
authors  thank Prof. Boling Guo for his constant encouragement and helpful 
suggestions. {The authors thank the anonymous reviewers for their helpful comments.}

\end{document}